\documentclass[a4paper,UKenglish,cleveref]{lipics-v2021}

\usepackage{mathtools}
\usepackage{physics}
\usepackage{graphicx}
\usepackage{multicol}
\usepackage[dvipsnames]{xcolor}
\usepackage{xspace}
\usepackage{cite}
\usepackage{bm}

\nolinenumbers
\hideLIPIcs

\DeclareGraphicsExtensions{.pdf,.png,.eps}

\newcommand{\N}{\mathbb{N}}
\newcommand{\calG}{\mathcal{G}}
\newcommand{\strongprod}{\mathbin{\boxtimes}}
\newcommand{\titlemath}[1]{%
    \texorpdfstring{$\bm{#1}$}{#1}%
}

\NewDocumentCommand{\geqtwosubdiv}{o}{%
    $\succcurlyeq \! 2$-subdivision\IfValueT{#1}{#1}%
}

\DeclarePairedDelimiterX\set[1]\lbrace\rbrace{\def\given{\;\delimsize\vert\;}#1}
\DeclarePairedDelimiter\floor\lfloor\rfloor
\DeclareMathOperator{\rtw}{rtw}
\DeclareMathOperator{\tw}{tw}
\DeclareMathOperator{\col}{scol}
\DeclareMathOperator{\Reach}{SReach}
\DeclareMathOperator{\dist}{dist}

\newtheorem{question}{Question}

\newcommand{\Dujmovic}{Dujmovi\'{c}\xspace}

\title{The \texorpdfstring{$\bm{r}$}{r}-Dynamic Chromatic Number is Bounded in the Strong 2-Coloring Number}

\author{Miriam Goetze}{Karlsruhe Institute of Technology, Germany}{miriam.goetze@student.kit.edu}{https://orcid.org/
0000-0001-8746-522X}{funded by the Deutsche Forschungsgemeinschaft (DFG, German Research Foundation) -- 520723789}

\author{Torsten Ueckerdt}{Karlsruhe Institute of Technology, Germany}{torsten.ueckerdt@kit.edu}{https://orcid.org/0000-0002-0645-9715}{}

\authorrunning{M. Goetze and T. Ueckerdt}

\Copyright{Miriam Goetze, Torsten Ueckerdt}



\begin{document}

\maketitle

\begin{abstract}
    A proper vertex-coloring of a graph is $r$-dynamic if the neighbors of each vertex $v$ receive at least $\min(r, \deg(v))$ different colors.
    In this note, we prove that if $G$ has a strong $2$-coloring number at most $k$, then $G$ admits an $r$-dynamic coloring with no more than $(k-1)r+1$ colors.
    As a consequence, for every class of graphs of bounded expansion, the $r$-dynamic chromatic number is bounded by a linear function in $r$.
    We give a concrete upper bound for graphs of bounded row-treewidth, which includes for example all planar graphs.
\end{abstract}

\section{Introduction}

In this paper we investigate $r$-dynamic colorings of graphs.
For a graph $G$ with vertex-set $V(G)$ and edge-set $E(G)$, and integers $k,r \geq 0$, a $k$-coloring $\varphi \colon V(G) \to [k]$ of the vertices of $G$ is called \emph{$r$-dynamic} if $\varphi$ is a proper coloring of $G$, every vertex $v$ of degree $\deg(v) \geq r$ has neighbors of at least~$r$ different colors, and for every vertex $v$ of degree $\deg(v) < r$ no color in its neighborhood $N_G(v)$ appears twice.
In other words, $\varphi$ is $r$-dynamic if

\begin{itemize}
    \item for every edge $uv \in E(G)$ we have $\varphi(u) \neq \varphi(v)$ and
    \item for every vertex $v \in V(G)$ we have $\abs{\varphi(N_G(v))} \geq \min(r,\deg_G(v))$.
\end{itemize}

The minimum $k$ such that $G$ admits an $r$-dynamic $k$-coloring is called the \emph{$r$-dynamic chromatic number} of $G$, denoted by~$\chi_r(G)$.
Clearly, for every graph $G$ with chromatic number $\chi(G)$ and maximum degree $\Delta(G)$ it holds
\begin{equation}
    \chi(G) = \chi_0(G) \leq \chi_1(G) \leq \cdots \leq \chi_{\Delta(G)}(G) = \chi(G^2),
\end{equation}

where $G^2$ is obtained from $G$ by connecting each pair of vertices at distance at most $2$ by an edge in $G$.
In fact, for every $r \geq \Delta(G)$ the $r$-dynamic colorings of $G$ correspond exactly to so-called $2$-distance colorings of~$G$, i.e., vertex-colorings where any two vertices with distance at most~$2$ have distinct colors.
For a survey on general $d$-distance colorings, let us refer to~\cite{kramer_survey_distance-coloring_2008}.
For now, let us be content to noting that $\chi(G^2) \leq \Delta(G)^2 + 1$ for every graph $G$ and thus all $r$-dynamic chromatic numbers $\chi_r(G)$ of $G$ are bounded in terms of its maximum degree $\Delta(G)$ independent of $r$.

In this paper we shall be interested in bounding $\chi_r(G)$ solely in terms of $r$, even though $\Delta(G)$ can be arbitrarily large.
(Broadly speaking, the reader may think of small $r$ and large $\Delta(G)$.)
For general graphs $G$, $\chi_r(G)$ cannot be bounded in terms of $r$ and $\chi(G)$ only.
A very instructive example due to Lai et al.~\cite{lai_conditional_colorings_2006} is the complete subdivision $K'_n$ of $K_n$, as shown in \cref{fig:examples}.
Even though~$K'_n$ is bipartite, we have $\chi_r(K'_n) \geq n$ for every $r \geq 2$ and $n \geq 4$.
Indeed, if two degree-$(n-1)$ vertices $u,v$ receive the same color, their common neighbor $w$ of degree~$2$ sees only $1 < \min(\deg(w),r)$ colors, making the coloring not $r$-dynamic.

\begin{figure}[htb]
    \centering
    \includegraphics[page=3]{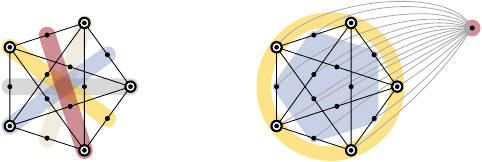}
    \caption{
        Left: A $2$-dynamic $5$-coloring of the graph $K'_5$ obtained from $K_5$ (large vertices) by adding a subdivision vertex (small) on every edge.
        Right: Adding one universal vertex allows for a $2$-dynamic $3$-coloring.
    }
    \label{fig:examples}
\end{figure}

Extending this example shows that adding vertices (or edges) can lead to a decrease of the $r$-dynamic chromatic number.
In fact, adding a universal vertex to $K'_n$ immediately lets the $2$-dynamic chromatic number drop down to~$3$.
For $r \geq 3$, $K'_n$ with $r-2$ universal vertices has $r$-dynamic chromatic number $n+r-2$, while with one further universal vertex that $r$-dynamic chromatic number drops down to $r+1$, as observed by Miao et al.~\cite{miao_element_deletion_dynamic_coloring_2016}.

\medskip

The graph class $\mathcal{G} = \{K'_n \mid n \geq 3\}$ used above is a classical example for a class of sparse bipartite graphs that still requires arbitrarily many colors in various other coloring variants (like star chromatic number, odd chromatic number, or acyclic chromatic number).
One reason is that also many structural parameter (like treewidth, genus, or Hadwiger number) are unbounded on $\mathcal{G}$.
Here we shall show that for nice, structured graph classes, the $r$-dynamic chromatic numbers are indeed bounded (in terms of $r$).
This includes for example planar graphs, graphs of bounded treewidth, or proper minor-closed graph classes \cite{nevsetvril_grad_2008}\footnote{We show that the $r$-dynamic chromatic number is bounded for every graph class of bounded expansion, which includes all of these graph classes. See \cite[Figure~$6$]{nevsetvril_from_sparse_2010} for an overview.}.

\subparagraph*{Previous Results}
The concept of $r$-dynamic colorings was initially studied for $r=2$~\cite{montgomery_dynamic_colorings_2001,lai_upper_bounds_2-dyn_2003} under the name dynamic colorings, and later generalized for general values of~$r$~\cite{lai_conditional_colorings_2006}.
In fact,~$r$-dynamic colorings are also known as $r$-hued colorings and conditional colorings, and by now there is a plethora of work on this concept.
Let us simply refer to the recent and very detailed survey of Chen et al.~\cite{chen_survey_2022} and the more than 100 references therein.

Closely related are $2$-distance colorings, as already mentioned above.
In 1977, Wegner studied $2$-distance colorings of planar graphs and conjectured that for every planar graph~$G$ with~$\Delta(G)\geq 8$ we have~$\chi_{\Delta(G)} \leq \floor{\frac{3\Delta(G)}{2}}+1$~\cite{wegner_planar_1977}. 
Since then, results have been obtained which are close to the estimated upper bound~\cite{havet_list_colouring_squares_2007}.
Song and Lai consider a generalization of Wegner's conjecture to $r$-dynamic colorings of planar graphs, and give an upper bound on~$\chi_r(G)$ that is linear in~$r$.

\begin{theorem}[{Song and Lai~\cite[Theorem~1.2]{song_planar_2018}}]{\ \\}%
    \label{thm:r-dyn_for_planar}%
    For $r \geq 8$ and every planar graph~$G$, we have $\chi_r(G) \leq 2r+16$.
\end{theorem}

This can be generalized to classes of graphs with bounded genus.
For every graph with genus~$\gamma$ its $r$-dynamic chromatic number lies in $O_\gamma(r)$.

\begin{theorem}[{Loeb et al.~\cite[Theorem~1.5]{loeb_genus_2018}}]{\ \\}%
    \label{thm:r-dyn_in_genus}%
    For every $r \geq 0$ and every graph~$G$ of genus~$\gamma$, we have
    \[
        \chi_r(G) \leq
        \begin{cases}
            (r+1)(\gamma+5)+3, &\text{ for }\gamma \leq 2 \text{ and } r \geq 2\gamma+11\\
            (r+1)(2\gamma+2)+3, &\text{ for }\gamma \geq 3 \text{ and } r \geq 4\gamma+5.
        \end{cases}
    \]
\end{theorem}

A further generalization are graphs which contain no~$K_t$ (for fixed $t$) as a minor.
For such graphs, explicit upper bounds on the $2$-dynamic chromatic number have been obtained~\cite{kim_k5-minor-free_2016}.

\subparagraph*{Our results}

We consider a graph parameter which is known as the \emph{strong $2$-coloring number} and denoted by $\col_2(G)$.
We show that for every $r \in \N$ and every graph $G$ the $r$-dynamic chromatic number $\chi_r(G)$ is bounded in terms of $r$ and $\col_2(G)$.
This implies that the $r$-dynamic chromatic number is bounded for every graph class of bounded expansion.
As this in particular includes graphs of bounded genus, this generalizes \cref{thm:r-dyn_for_planar,thm:r-dyn_in_genus}.
Further, we obtain for all integers~$t$ explicit upper bounds on the strong $t$-coloring number in terms of the row-treewidth, which may be of independent interest.
This yields in particular explicit bounds on the $r$-dynamic chromatic number.

\begin{theorem}
    \label{thm:corollary}
    Let~$G$ be a graph and $r \in \N$.
    \begin{enumerate}
        \item If~$G$ has strong $2$-coloring number at most~$k$, then $\chi_r(G) \leq (k-1)r+1$.
        \item If~$G$ has treewidth at most~$k$, then $\chi_r(G) \leq kr+1$.
        \item If~$G$ has row-treewidth at most~$k$, then $\chi_r(G) \leq 5kr+1$.
    \end{enumerate}
\end{theorem}

\subparagraph*{Organization of the paper}

We define the necessary notions, like strong coloring numbers, bounded expansion, and (row-)treewidth in \cref{sec:preliminaries}.
In \cref{sec:main-result}, we prove our main result, cf.~\cref{thm:r-dyn_bounded_by_2-coloring_number}.
We give upper bounds on the strong $2$-coloring number in terms of the graph's (row-)treewidth in \cref{sec:applications}, which then proves \cref{thm:corollary}.
We close with a short discussion in \cref{sec:conclusion}.

\section{Preliminaries}
\label{sec:preliminaries}

\subparagraph*{Strong coloring numbers}

Let~$G$ be a graph and consider a linear order~$\pi$ of the vertices of~$G$. 
Let~$v$ be a vertex of~$G$ and let $k \in \N$. 
We say that a vertex $u$ is \emph{$k$-accessible} by~$v$ with respect to~$\pi$ if $u \leq v$ and there exists a $u$-$v$-path of length at most~$k$ such that $v < w$ for all inner vertices~$w$ of~$P$.
(The comparisons $u \leq v$ and $v < w$ are with respect to the linear order $\pi$.)
We denote by~$\Reach_k(v,\pi)$ the set of all vertices of~$G$ that are $k$-accessible by~$v$ with respect to~$\pi$. 
Note that $v \in \Reach_k(v,\pi)$.
When the linear order $\pi$ is clear from context, we may write $\Reach_k(v)$ instead.

The \emph{strong $k$-coloring number} of~$G$ is defined as the minimum size of the largest set~$\Reach_k(v,\pi)$ over all linear orders~$\pi$, i.e.,
\[
    \col_k(G) = \min_{\pi \in \Pi(G)} \max_{v \in V(G)} \abs{\Reach_k(v,\pi)},
\]
where~$\Pi(G)$ denotes the set of all linear orders of the vertices of~$G$.

Kierstead and Yang introduce the strong $k$-coloring number in~\cite{kierstead_orderings_2003}.
They show in particular that the \emph{acyclic chromatic number} is bounded by the strong $2$-coloring number, i.e., if~$G$ is a graph with $\col_2(G) \leq c$, then there exists a proper vertex $c$-coloring of~$G$ where any two color classes induce a forest.
The strong $2$-coloring number is similar to the notion of \emph{arrangeability} considered in
\cite{chen_graphs_linearly_bounded_ramsey_1993, rodl_arrangeability_2013, kierstead_planar_graph_coloring_1994} and the \emph{rank} of a graph introduced in~\cite{kierstead_simple_competitive_2000}.
It has been applied to bound many other coloring numbers such as the \emph{odd coloring number} and the \emph{conflict-free coloring number}~\cite{hickingbotham_odd-colourings_2022}, the \emph{game chromatic number} and \emph{oriented game chromatic number}~\cite{kierstead_planar_graph_coloring_1994, kierstead_competitive_2001}, and the \emph{boxicity}~\cite{esperet_boxicity_2018}\footnote{In fact, the boxicity is bounded in terms of the \emph{weak $2$-coloring number}, which however is bounded in terms of the strong $2$-coloring number~\cite[Lemma 2.7]{zhu_colouring_2009}.}.
We will use the strong $2$-coloring number to bound the $r$-dynamic chromatic number.

\subparagraph*{Bounded Expansion}

The concept of a graph class having bounded expansion is introduced in~\cite{nevsetvril_grad_2008}.
It provides one way of formalizing the sparsity of graphs in the class, while neither being the only nor the most general such formalization.
There are many different characterizations of bounded expansion, see~\cite{nevsetvril_sparsity_2012} for a detailed survey on sparsity.
For our purposes, we may consider the following theorem as the definition of \emph{bounded expansion}.

\begin{theorem}[{Zhu~\cite{zhu_colouring_2009}}]{\ \\}%
    \label{thm:bounded_expansion_iff_bounded_coloring_number}%
    A graph class~$\calG$ has bounded expansion if and only if for all~$k \in \N$ there exists a constant~$d = d(k)$ such that $\col_k(G) \leq d$ for all $G \in \calG$. 
\end{theorem}

Kierstead and Yang observe that for every topologically closed graph class~$\calG$ with bounded strong $1$-coloring number (which is generally referred to as the coloring number), also all strong $k$-coloring numbers with $k \geq 2$ are bounded~\cite[Theorem~4]{kierstead_orderings_2003}. 
As topologically closed graph classes have bounded expansion, \cref{thm:bounded_expansion_iff_bounded_coloring_number} can be considered as a generalization of their result.

Many sparse graph classes have bounded expansion.
Important examples include planar graphs, graphs with bounded queue number or bounded stack number, and topologically closed graph classes, see~\cite[Fig.~1]{nevsetvril_characterisations_2012} for an overview.
According to \cref{thm:bounded_expansion_iff_bounded_coloring_number}, once we bound the $r$-dynamic chromatic number in the strong $2$-coloring number, it follows that the $r$-dynamic chromatic number is bounded for all graph classes with bounded expansion.

\subparagraph*{Treewidth}

Treewidth is a parameter that captures the similarity of a graph to a tree.
We define it through the notion of \emph{$k$-trees}, which in turn are recursively defined as follows.
For any integer $k \geq 0$, a $(k+1)$-clique is a $k$-tree.
Moreover, if~$G$ is a $k$-tree, then any graph that is obtained by introducing one new vertex~$v$ to $G$ and connecting $v$ (by an edge) to all vertices of a $k$-clique in~$G$ is also a $k$-tree.
Note in particular that any graph~$G$ on~$n$ vertices is a subgraph of an~$(n-1)$-tree. 
The minimum~$k$ such that~$G$ is a subgraph of some $k$-tree is called the \emph{treewidth} of $G$ and denoted by $\tw(G)$.

If a graph $G$ is a $k$-tree, then there exists a linear order $\pi$ of the vertices of $G$ such that for every vertex~$v$ in $G$ the right neighborhood $N^+(v) = \{u \in V(G) \mid uv \in E(G), u > v\}$ forms a clique.
Such an ordering $\pi$ is called a \emph{perfect elimination ordering} of $G$.
Recall that a $k$-tree $G$ is inductively constructed by starting with $k+1$ vertices forming a clique, and then adding one vertex at a time.
Note that the reverse order of this construction sequence is a perfect elimination ordering of $G$.

\subparagraph*{Row-Treewidth}

The \emph{strong product} $H \boxtimes H'$ of two graphs~$H,H'$ is the graph with vertex-set~$V(H) \times V(H')$ with an edge between two vertices~$(u,x),(v,y)$ if and only if
\[ 
\text{$u=v,xy \in E(H')$ \quad or \quad $x = y,uv \in E(H)$ \quad or \quad $uv \in E(H),xy \in E(H')$}.
\]
Now suppose one of the graphs is a path~$P$ and~$x_1, \dots, x_n$ is the order of the vertices on~$P$.
Setting~$L_i = \set{(u,x_i) \given u \in V(H)}$ for $i \in [n]$, we obtain a layering~$L_1, \dots, L_n$ of $H \strongprod P$. 
Note that only vertices of layers~$L_i$ and~$L_j$ with $\abs{i-j} \leq 1$ may be adjacent in~$H \strongprod P$.
Each layer corresponds to a copy of~$H$.

The minimum~$k$ such that a graph~$G$ is a subgraph of the strong product~$H \strongprod P$ for some $k$-tree~$H$ and some path~$P$ is the \emph{row-treewidth} of~$G$ and denoted by~$\rtw(G)$.
It has (implicitly) been introduced in \cite{dujmovic_queue_2020} and provides a tool to generalize results known for graph classes of bounded treewidth to some minor-closed graph classes, such as planar graphs~\cite{DGLTU-22,BDJM-22,BKW-22,DEGJMM-21,DEMW-23,DFMS-20}.

\section{Bounding the \texorpdfstring{\titlemath{r}}{r}-chromatic number in the strong 2-coloring number}
\label{sec:main-result}

Let us prove our main result, namely that for every graph $G$ and integer $r \geq 1$, $\chi_r(G)$ is bounded in terms of $r$ and $\col_2(G)$.

\begin{theorem}
    \label{thm:r-dyn_bounded_by_2-coloring_number}
    For every $r,k \in \N$, if~$G$ is a graph with $\col_2(G) \leq k$, then $\chi_r(G) \leq (k-1)r+1$.
\end{theorem}
\begin{proof}
    Consider a linear order $v_1,v_2, \dots, v_n$ of the vertices of~$G$ such that $\abs{\Reach_2(v)} \leq k$ for every vertex~$v$.
    For every~$i \in [n]$, we denote by~$G_i$ the subgraph of $G$ induced by the first~$i$ vertices.
    
    We say that a vertex coloring~$\varphi$ of~$G_i$ is \emph{strongly proper} if no vertex~$v \in V(G_i)$ has the same color as any vertex $u \in \Reach_2(v) - \{v\}$.
    Note that a strongly proper coloring is in particular a proper coloring.
    If $v$ is a vertex of~$G_i$, we say that a coloring~$\varphi$ is \emph{weakly $r$-dynamic at~$v$} if $\abs{\varphi(N_{G_i}(v))} \geq \min(r, \abs{N_{G_i}(v)})$, that is, if the number of colors in the neighborhood of $v$ in $G_i$ is at least $\min(r, \abs{N_{G_i}(v)})$.

    Let $d \coloneqq (k-1)r+1$.
    We iteratively extend a weakly $r$-dynamic, strongly proper coloring~$\varphi_i$ of~$G_i$ that only uses colors in~$[d]$ to a coloring~$\varphi_{i+1}$ of~$G_{i+1}$ with the same properties.
    Using $d$ distinct colors on the first~$d$ vertices $v_1,\ldots,v_d$, we obtain a coloring $\varphi_d$ that is weakly $r$-dynamic and strongly proper on $G_d$.

    Suppose $i \geq d$ and we already have the desired coloring~$\varphi_i$ of~$G_i$ which only uses colors in~$[d]$.
    We need to assign a color in~$[d]$ to the vertex~$v_{i+1}$ such that the resulting coloring~$\varphi_{i+1}$ is weakly $r$-dynamic and strongly proper on~$G_{i+1}$.
    To this end, it suffices to choose a color for~$v_{i+1}$ such that the following three properties are satisfied for $\varphi_{i+1}$ on $G_{i+1}$:
    
    \begin{enumerate}
        \item\label{prop:coloring_number_r-dyn_at_v_n} $\varphi_{i+1}$ is weakly $r$-dynamic at~$v_{i+1}$
        \item\label{prop:coloring_number_strongly_proper} $\varphi_{i+1}$ is strongly proper
        \item\label{prop:coloring_number_r-dyn_for_neighbors} $\varphi_{i+1}$ is weakly $r$-dynamic at every vertex in $N_{G_{i+1}}(v_{i+1})$.
    \end{enumerate}
    
    Observe that property~\ref{prop:coloring_number_r-dyn_at_v_n} is always satisfied, since the neighborhood of $v_{i+1}$ in $G_{i+1}$ is a clique and hence all vertices in $N_{G_{i+1}}(v_{i+1})$ have distinct colors as $\varphi_i$ is proper.
      
    For properties~\ref{prop:coloring_number_strongly_proper} and~\ref{prop:coloring_number_r-dyn_for_neighbors}, let us call a color~$c$ \emph{forbidden} for a given property if the property does not hold for the extension of $\varphi_i$ in which $v_{i+1}$ is colored in~$c$.
    Let~$n_2$ and~$n_3$ denote the number of colors forbidden by property~\ref{prop:coloring_number_strongly_proper} and~\ref{prop:coloring_number_r-dyn_for_neighbors} respectively.
    We claim that~$n_2 + n_3 < d$.
    
    First, we show that~$n_2 \leq k-1$.
    In order to preserve property~\ref{prop:coloring_number_strongly_proper} of being strongly proper, we may not use for $v_{i+1}$ any of the colors that appears on a vertex in $\Reach_2(v_{i+1})-\set{v_{i+1}}$.
    Yet, $\abs{\Reach_2(v_{i+1})} \leq k$ and hence $n_2 \leq k-1$.

    Now, we show that~$n_3 \leq (k-1)(r-1)$.
    If a vertex $v \in N_{G_{i+1}}(v_{i+1})$ has at most~$r-1$ colors in its neighborhood in~$G_i$, all these colors are forbidden for~$v_{i+1}$.
    If $v$ has at least~$r$ colors in its neighborhood in~$G_i$, then no matter which color we use on~$v_{i+1}$, the resulting coloring is weakly $r$-dynamic at~$v$. 
    Thus, for each neighbor~$v \in N_{G_{i+1}}(v_{i+1})$, we exclude at most~$r-1$ colors as forbidden for property~\ref{prop:coloring_number_r-dyn_for_neighbors}.
    As $N_{G_{i+1}}(v_{i+1}) \subseteq \Reach_2(v_{i+1}) - \{v_{i+1}\}$, it follows that $n_3 \leq (r-1)\abs{N_{G_{i+1}}(v_{i+1})} \leq (r-1)(\abs{\Reach_2(v_{i+1})}-1) \leq (r-1)(k-1)$.
    
    Thus, $n_2+n_3 \leq (k-1)r < d$, and hence there is a color in $[d]$ for~$v_{i+1}$ such that the resulting coloring $\varphi_{i+1}$ on $G_{i+1}$ fulfills all three properties, which concludes the proof.
\end{proof}

By \cref{thm:r-dyn_bounded_by_2-coloring_number}, $\chi_r(G)$ is bounded in terms of $\col_2(G)$ and $r$.
On the other hand, $\chi_r(G)$ is in general not bounded in terms $\col_1(G)$ and $r$:
Subdividing every edge of the complete graph~$K_n$ yields a bipartite graph~$K_n'$. 
We first observe that $\chi_r(K_n') \geq n$ for $r \geq 2$. 
Indeed, no two original vertices~$u,v$ of~$K_n$ can have the same color as otherwise the subdivision vertex of the edge~$uv$ would have only one color in its neighborhood.
Yet, $\col_1(K_n') \leq 3$, which is certified by any linear order in which all original vertices appear before the subdivision vertices.

\section{Applying \texorpdfstring{\cref{thm:r-dyn_bounded_by_2-coloring_number}}{Theorem 5} and explicit bounds}
\label{sec:applications}

In this section we shall apply \cref{thm:r-dyn_bounded_by_2-coloring_number} to several graph classes $\calG$ and thereby obtain upper bounds on the $r$-dynamic chromatic number of any graph $G \in \calG$.
For starters, recall from \cref{thm:bounded_expansion_iff_bounded_coloring_number} that for any class $\calG$ of bounded expansion the strong $2$-coloring numbers of graphs in $\calG$ are bounded (by an absolute constant).
Thus, \cref{thm:r-dyn_bounded_by_2-coloring_number} immediately gives the following.

\begin{corollary}
    \label{cor:r_dyn_bounded_for_bounded_expansion}
    For every graph class $\calG$ of bounded expansion and every $r \in \N$, the $r$-dynamic chromatic numbers of graphs in $\calG$ is bounded by a linear function in $r$.
    That is, there exists an absolute constant $c = c(\calG)$ such that for every graph $G \in \calG$ we have $\chi_r(G) \leq c \cdot r$.
\end{corollary}

The constant $c = c(\calG)$ in \cref{cor:r_dyn_bounded_for_bounded_expansion} can be chosen as $\col_2(\calG) = \max\{ \col_2(G) \mid G \in \calG \}$.

In the following, we give upper bounds on $\col_2(\calG)$ for the case that $\calG$ has bounded treewidth or bounded row-treewidth.
In fact, let us present upper bounds on the strong $t$-coloring number of graphs in $\calG$ for all $t \in \N$.

Recall that $\col_t(G)$ is the largest number of vertices that are $t$-accessible from the same vertex $v \in V(G)$, minimized over all linear orders of $V(G)$.
Given a linear order, any vertex that is $t$-accessible from vertex~$v$ is also $(t+1)$-accessible from~$v$. 
Thus, we have
\[
    \col_1(G) \leq \col_2(G) \leq \dots \leq \col_n(G),
\]
where $n = \abs{V(G)}$.
In fact, the strong $t$-coloring numbers of $G$ interpolate between the coloring number (i.e. the strong $1$-coloring number) and the treewidth of $G$. 
If~$G$ has $n$~vertices, then $\col_{n}(G) = \tw(G)+1$~\cite[p. 2469]{grohe_coloring_2018}.
For the sake of completeness, we also give a proof for the statement.

\begin{lemma}
    \label{lem:strong_t_col_bounded_in_treewidth}
    For every graph $G$ and every $t \in \N$ we have $\col_t(G) \leq \tw(G)+1$.
\end{lemma}
\begin{proof}
    Let $k = \tw(G)$ be the treewidth of $G$.
    As the strong $t$-coloring number is monotone with respect to taking subgraphs, we may assume that~$G$ is a $k$-tree.
    Consider the reverse order of a perfect elimination ordering, i.e., a linear order $\pi$ such that for every vertex $v$, its left neighborhood $N^-(v) = \{u \in V(G) \mid uv \in E(G), u < v\}$ forms a clique in $H$.
    Then also $\Reach_1(v) = N^-(v) \cup \{v\}$ forms a clique, and its size is at most $k+1$.
    This already proves that $\col_1(G) \leq k+1$.

    For $t \geq 2$, we claim that $\Reach_t(v) \subseteq \Reach_1(v)$ for all~$v \in V(G)$.
    To this end, consider the graph induced by all vertices~$y \in V(G)$ with $x \leq y$ and let~$G'$ be the component which contains~$x$. 
   
    \begin{claim}
    \label{claim:adj_of_accessible_vertices}
    If~$y$ is vertex in~$G'$ and~$w$ is any neighbor of~$y$ in~$G$ with~$w < x$, then $w$ is adjacent to~$x$.
    \end{claim}
    \begin{claimproof}
    We show the claim by induction on~$k = \dist_{G'}(x,y)$.
    The claim clearly holds for all vertices in~$G'$ at distance~$k=0$ from~$x$. 
    Suppose the claim holds for all vertices at distance less than~$k$ for some integer~$k$ and consider a vertex~$y$ in~$G'$ with~$\dist_{G'}(x,y)=k$. 
    Let~$w$ be a neighbor of~$y$ in~$G$ with~$w < x$.
    Consider a shortest~$x$-$y$-path~$P$ in~$G'$ and let~$z$ be the neighbor of~$y$ on~$P$.
    
    We first note that the vertices of~$P$ are increasing with respect to the linear order. 
    Suppose otherwise and let~$u$ be the rightmost vertex on~$P$. 
    By assumption, $u$ is an inner vertex of the path.
    Let~$u'$ and~$u''$ be its two neighbors on~$P$. 
    We have $u', u'' < u$. 
    Thus, $u'$ and $u''$ are adjacent, contradicting the fact that~$P$ is a shortest path. 

    Therefore, we have in particular $z < y$.
    As we considered the inverse of a perfect elimination ordering and both~$w$ and~$z$ are neighbors of~$y$, the two vertices are adjacent.
    By induction, $x$ is adjacent to~$w$ as~$\dist_{G'}(x,z) = k-1$. 
    \end{claimproof}

    We now show that $\Reach_t(x) \subseteq \Reach_1(x)$. 
    Let $w \in \Reach_t(x)$, and consider the corresponding $x$-$w$-path~$P$ where $x < u$ for all inner vertices~$u$.
    If $w = x$, we clearly have $w \in \Reach_1(x)$. 
    Otherwise, we need to show that~$w$ is a neighbor of~$x$. 
    This follows from the claim above.
    Indeed, if we denote by~$y$ the neighbor of~$w$ on~$P$, then the conditions of \cref{claim:adj_of_accessible_vertices} are met.
    Thus, $wx$ is an edge and we obtain $w \in \Reach_1(x)$.
\end{proof}

Now, \cref{lem:strong_t_col_bounded_in_treewidth} gives $\col_2(G) \leq \tw(G)+1$ and \cref{thm:r-dyn_bounded_by_2-coloring_number} gives $\chi_r(G) \leq (\col_2(G)-1)r+1$, which together gives the following.

\begin{corollary}
\label{cor:r-dyn_in_tw}
    For every graph $G$ and every $r \in \N$ we have $\chi_r(G) \leq \tw(G)\cdot r+1$.
\end{corollary}

Turning from treewidth to bounded genus, Van den Heuvel and Wood~\cite[p. 21-22]{vandenheuvel_improper_arxiv_2018} prove upper bounds on the strong $t$-coloring numbers of graphs in terms of their genus.
More generally, they consider \emph{$(g,k)$-planar} graphs, which are graphs that can be embedded on a surface of genus~$g$ such that each edge is crossed at most~$k$ times, and obtain the following result\footnote{The paper has already appeared in a journal~\cite{vandenheuvel_improper_2018}. Yet, $(g,k)$-planar graphs are only considered in the arXiv-version.}.

\begin{lemma}[{Van den Heuvel and Wood~\cite[p. 22]{vandenheuvel_improper_arxiv_2018}}]{\ \\}%
    \label{lem:col_2_bounded_for_g_k_planar}%
    For every $(g,k)$-planar graph $G$ and every $t \in \N$ we have $\col_t(G) \leq (2t+1)(4g+6)(k+1)$.
\end{lemma}

Plugging this bound on $\col_2(G)$ again into \cref{thm:r-dyn_bounded_by_2-coloring_number}, we can bound the $r$-dynamic chromatic number of $(g,k)$-planar graphs.

\begin{corollary}
    \label{cor:r_dyn_bounded_for_g_k_planar}
    For every $(g,k)$-planar graph $G$ and every $r \in \N$ we have $\chi_r(G) \leq 5(4g+6)(k+1)r+1$.
\end{corollary}

\begin{remark}
    \label{remark:r_dyn_bounded_in_genus}
    \cref{cor:r_dyn_bounded_for_g_k_planar} yields an upper bound on the $r$-dynamic chromatic number of planar graphs that is linear in $r$, i.e., a similar result to \cref{thm:r-dyn_for_planar}.
    Yet, the constant factor from \cref{cor:r_dyn_bounded_for_g_k_planar} is worse.
    Similarly, \cref{cor:r_dyn_bounded_for_g_k_planar} yields an upper bound on the $r$-dynamic chromatic number that is linear in $r$ and the genus of the graph, i.e., a similar result to \cref{thm:r-dyn_in_genus}.
\end{remark}

Let us also remark that \cref{lem:col_2_bounded_for_g_k_planar} is obtained by an upper bound on the strong $t$-coloring number in terms of the so-called \emph{layered treewidth} of the graph, which is a notion related to row-treewidth~\cite[Lemma~30]{vandenheuvel_improper_arxiv_2018}. 
We show a similar result for row-treewidth.
Layered treewidth is bounded by row-treewidth.
Yet, as the row-treewidth in general is not bounded in terms of the layered treewidth~\cite{bose_separating_layered_treewidth_2022}, our result (\cref{lem:strong_t_coloring_number_bounded_by_row_treewidth}) is a generalization of~\cite[Lemma~30]{vandenheuvel_improper_arxiv_2018}.

\begin{lemma}
    \label{lem:strong_t_coloring_number_bounded_by_row_treewidth}
    For every graph $G$ and every $t \in \N$ we have $\col_t(G) \leq (2t+1)(\rtw(G)+1)$.
\end{lemma}
\begin{proof}
    Let $k = \rtw(G)$ be the row-treewidth of $G$.
    Then there exists a path~$P$ and a $k$-tree~$H$ such that $G \subseteq H \strongprod P$. 
    We denote by~$\pi_H$ the linear order $x_1, x_2, \dots, x_n$ of the vertices of~$H$ such that the set of
    $t$-accessible vertices from~$x_i$ has size at most~$k+1$ for every~$x_i$, see \cref{lem:strong_t_col_bounded_in_treewidth}.
    
    Let~$A_x$ denote the copies of a vertex $x \in V(H)$ in $H \strongprod P$.
    Consider a linear order~$\pi$ of the vertices of~$H \strongprod P$ where 
    \[
        A_{x_1} \leq A_{x_2} \leq \dots \leq A_{x_n}.
    \]

    Let~$v$ be a vertex of~$H$ and let~$v_i$ denote its copy in the $i$-th layer.
    It suffices to show that $\abs{\Reach_t(v_i,\pi)} \leq (2t+1)(k+1)$ for every $i$. 
    It then follows that $\col_t(G) \leq (2t+1)(k+1)$.
    
    We show the upper bound on the size of $\abs{\Reach_t(v_i,\pi)}$ by splitting the set of $t$-accessible vertices of~$v_i$ into several sets, each of which contains only vertices of one layer.
    As the restriction~$\pi_i$ of the linear order to the $i$-th layer~$L_i$ corresponds to~$\pi_H$ and the vertices of~$L_i$ induce a copy of~$H$, at most $k+1$ vertices are $t$-accessible from~$v_i$ within~$L_i$, i.e. we have
    $\abs{\Reach_t(v_i,\pi_i)} \leq k+1$.

    We now show that 
    \begin{align}
        \label{eq:reach_t_for_product_structure}
        \Reach_t(v_i,\pi) \subseteq \bigcup_{\ell=i-t}^{i+t} \Reach_t(v_{\ell}, \pi_{\ell}).
    \end{align}
    The claim then follows as $\abs{\Reach_t(v_{\ell},\pi_{\ell})} \leq k+1$ for all~$\ell$. 

    Let~$u_j \in \Reach_t(v_i)$ and consider a corresponding $v_i$-$u_j$ path. 
    Projecting the path into~$H$ yields a walk in~$H$ of length at most~$t$ (possibly with loops as the path might visit two copies of the same vertex in~$H$).
    This induces a $v$-$u$-path in~$H$ where all inner vertices are to the right of~$v$ with respect to the linear order.
    Thus, we clearly have 
    \[
        u_j \in \Reach_t(v_j,\pi_j).
    \]
    
    The vertex~$u_j$ belongs to a layer~$L_{j}$ with $i-t \leq j \leq i+t$ as it has distance at most~$t$ to~$v_i$.
    Relation~\eqref{eq:reach_t_for_product_structure} follows.
\end{proof}

In particular, \cref{lem:strong_t_coloring_number_bounded_by_row_treewidth} gives that $\col_2(G) \leq 5\cdot \rtw(G)$, which we can again plug into \cref{thm:r-dyn_bounded_by_2-coloring_number}.

\begin{corollary}
    \label{cor:r-dyn_in_rtw}
    For every graph $G$ and every $r \in \N$ we have $\chi_r(G) \leq 5\cdot \rtw(G) \cdot r + 1$.
\end{corollary} 

\Dujmovic et al. showed that graphs of bounded genus, and in particular planar graphs have small row-treewidth~\cite[Theorem~37]{dujmovic_queue_2020}\footnote{\Dujmovic et al. prove that any graph of Euler genus~$g$ is is a subgraph of~$(K_{2g}+H)\strongprod P$ for some graph~$H$ with $\tw(H) \leq 8$ and some path~$P$.}.
Hence, \cref{cor:r-dyn_in_rtw} gives another upper bound on the $r$-dynamic chromatic number of graphs of genus~$g$ that is linear in~$r$ and~$g$; cf.\ \cref{remark:r_dyn_bounded_in_genus}.

\section{Discussion}
\label{sec:conclusion}

We provided an upper bound on the $r$-dynamic chromatic number in terms of the strong $2$-coloring number.
Upper bounds on the strong $2$-coloring number of many graph classes have already been obtained in previous work, in particular for $k$-planar graphs~\cite[p. 22]{vandenheuvel_improper_arxiv_2018}, see \cref{lem:col_2_bounded_for_g_k_planar}.
As the $r$-dynamic chromatic number is bounded in the strong $2$-coloring number, bounds on~$\chi_r$ follow for these graph classes.
In particular, we bounded~$\chi_r$ in the treewidth.
More precisely, we showed that for every graph~$G$ that $\chi_r(G) \leq \tw(G)\cdot r+1$ (cf. \cref{cor:r-dyn_in_tw}).
Yet, we know no example where the upper bound is attained.
\begin{question}
Is $\chi_r(G) \in O(\tw(G)+r)$ for every graph~$G$?
\end{question}

We have seen that the $r$-dynamic chromatic number is bounded for every graph class of bounded expansion (\cref{cor:r_dyn_bounded_for_bounded_expansion}).
One attempt of generalizing this result consists of bounding the $r$-dynamic chromatic number of nowhere dense graph classes which were introduced in~\cite{nevsetvril_nowhere_dense_2011}.
Yet, there are nowhere dense graph classes with unbounded chromatic number.
Consider for example the class~$\calG$ consisting of all graphs whose girth is larger than their maximum degree. 
This class is well-known to be nowhere dense, see for example~\cite[p.\,10]{nevsetvril_from_sparse_2010} and~\cite[p.\,4]{grohe_deciding_2017}.
Lubotzky, Phillips and Sarnak constructed (non-bipartite) $(p+1)$-regular graphs~$X^{p,q}$ whose girth grows with~$q$~\cite[p.\,263]{lubotzky_ramanujan_1988} and whose chromatic number is lower bounded by a function that only depends on~$p$~\cite[p.\,276]{lubotzky_ramanujan_1988}.
As the graphs~$X^{p,q}$ lie in~$\calG$ for large enough~$q$, 
the class~$\calG$ provides an example of a nowhere dense graph class with unbounded chromatic number.
In particular, the $r$-dynamic chromatic number of~$\calG$ is unbounded, since every $r$-dynamic coloring is proper, i.e., \cref{cor:r_dyn_bounded_for_bounded_expansion} cannot be generalized to nowhere dense graph classes.

Recall that for every $r\geq 2$, the $r$-dynamic chromatic number of the $1$-subdivision~$K_n'$ of a complete graph on $n$~vertices is at least~$n$.
Yet, subdividing every edge of~$K_n$ at least twice yields a graph with $r$-dynamic chromatic number at most~$2r+1$. 
In fact, this holds for every \emph{\geqtwosubdiv{}}, that is for every graph obtained by subdividing each edge of some graph at least twice.
\begin{lemma}
\label{lem:col_2_subdivisions}
    For every \geqtwosubdiv{}~$S$ of any graph $G$, we have $\col_2(S) \leq 3$.
    In particular, for all~$r \in \N$, $\chi_r(S) \leq 2r+1$.
\end{lemma}
\begin{proof}
    Consider a vertex-order of~$S$ where $u \leq x$ for every vertex~$u \in V(G)$ and every vertex~$x \in V(S) - V(G)$. 
    Observe that for every vertex~$u \in V(G)$, we have $\Reach_2(u) = \set{u}$.
    As every vertex~$x \in V(S) - V(G)$ has degree~$2$, we obtain $\abs{\Reach_2(x)} \leq 3$.
    Thus, $\col_2(S) \leq 3$ and $\chi_r(S) \leq 2r+1$ follows from \cref{thm:r-dyn_bounded_by_2-coloring_number}.
\end{proof}

While $\chi_r$ is bounded for every graph class of bounded expansion, there are also so-called~\cite{nevsetvril_from_sparse_2010} \emph{somewhere dense} graph classes with bounded $r$-dynamic chromatic number.
For example, the class~$\mathcal{S}$ of all \geqtwosubdiv{s} of complete graphs does not have bounded expansion (is not even nowhere dense~\cite[Section 2.1]{nevsetvril_nowhere_dense_2011}), but according to \cref{lem:col_2_subdivisions} still has $r$-dynamic chromatic numbers bounded by $2r+1$.

\bibliographystyle{plainurl}
\bibliography{bibliography}
\end{document}